\documentclass{amsart}
\usepackage{xcolor}
\usepackage{amsmath}
\usepackage{amsthm}
\usepackage{amssymb}
\usepackage{amsfonts}
\usepackage{hyperref}
\usepackage{amsfonts}
\usepackage{comment}
\usepackage{url}
\usepackage{graphicx}
\usepackage{xspace}
\usepackage{paralist}

\title{$\mathcal{I}^{\mathcal{K}}$-convergence}

\author{Martin Ma\v{c}aj}
\address{Department of Algebra, Geometry and Mathematical Education, Faculty of Mathematics, Physics and Informatics, Comenius University, Mlynsk\'a dolina, 842 48 Bratislava, Slovakia}
\email{macaj@dcs.fmph.uniba.sk}

\author{Martin Sleziak}
\email{sleziak@fmph.uniba.sk}

\newtheorem{PROP}{Proposition}[section]
\newtheorem{THM}[PROP]{Theorem}
\newtheorem{LM}[PROP]{Lemma}

\theoremstyle{definition}
\newtheorem{DEF}[PROP]{Definition}
\newtheorem{EXA}[PROP]{Example}
\newtheorem{REM}[PROP]{Remark}

\newcommand{\mc}[1]{\ensuremath{\mathcal{#1}}\xspace}

\newcommand{\Zobr}[3]{\ensuremath{#1\colon #2\to #3}}

\newcommand{\inv}[1]{\ensuremath{#1^{-1}}}
\newcommand{\Invobr}[2]{\inv{#1}(#2)}
\newcommand{\Obr}[2]{#1[#2]}

\newcommand{\ol}[1]{\ensuremath{\overline{#1}}}

\newcommand{\iaoi}{if and only if\xspace}
\newcommand{\Lra}{\ensuremath{\Leftrightarrow}\xspace}
\newcommand{\Ra}{\ensuremath{\Rightarrow}\xspace}
\newcommand{\emps}{\ensuremath{\emptyset}\xspace}
\newcommand{\abs}[1]{\lvert#1\rvert}

\newcommand{\sm}{\ensuremath{\setminus}}
\newcommand{\powerset}[1]{\mc P(#1)}

\newcommand{\I}{\ensuremath{\mathcal{I}}\xspace}

\newcommand{\J}{\ensuremath{\mathcal{K}}\xspace}
\newcommand{\FIh}[1]{\ensuremath{\mc F(#1)}}
\newcommand{\FI}{\ensuremath{\mc F(\I)}}
\newcommand{\FJ}{\ensuremath{\mc F(\J)}}
\newcommand{\Jlim}[1]{\operatornamewithlimits{#1\text{-}\lim}}
\newcommand{\Ilim}{\Jlim{\I}}
\newcommand{\Ihlim}{\Jhlim\I}
\newcommand{\Jhlim}[1]{#1^*\text{-}\lim}

\newcommand{\IhJ}{\IhJh{\I}{\J}}
\newcommand{\IhJh}[2]{\ensuremath{#1^{#2}}}
\newcommand{\IhJlim}{\operatorname{\IhJ\text{-}\lim}}
\newcommand{\IhJhlim}[2]{\operatorname{\IhJh{#1}{#2}\text{-}\lim}}
\newcommand{\Fin}{\ensuremath{\textrm{Fin}}}
\newcommand{\Imac}{\ensuremath{{\I_m}}}
\newcommand{\Iuni}{\ensuremath{{\I_1}}}

\newcommand{\Iprg}{\ensuremath{\I_2}}
\newcommand{\Bprg}{\ensuremath{{\mc B_2}}}

\newcommand{\subJJ}[1]{\subJb{\mc{#1}}}
\newcommand{\subJb}[1]{\ensuremath{\subset_{#1}}}
\newcommand{\subJ}{\subJJ{\J}}
\newcommand{\simJb}[1]{\sim_{#1}}
\newcommand{\simJ}{\simJb{\J}}

\newcommand{\APIJh}[2]{\ensuremath{\mathrm{AP(}#1,#2\mathrm{)}}}
\newcommand{\APIJ}{\ensuremath{\APIJh{\I}{\J}}}
\newcommand{\APIFin}{\ensuremath{\APIJh{\I}{\Fin}}}

\newcommand{\omone}{\ensuremath{\omega_1}\xspace}
\newcommand{\alone}{\ensuremath{\aleph_1}\xspace}

\newcommand{\seqq}[3]{({#1}_{#2})_{#3=1}^\infty}
\newcommand{\seq}[2]{\seqq{#1}{#2}{#2}}
\newcommand{\dseqq}[4]{({#1}_{#2})_{#3,#4=1}^\infty}
\newcommand{\dseq}[3]{\dseqq{#1}{#2,#3}{#2}{#3}}

\newcommand{\intrvl}[2]{\ensuremath{[{#1},{#2})}}

\newcommand{\R}{\mathbb R}
\newcommand{\N}{\mathbb N}

\newcommand{\enu}{\renewcommand{\theenumi}{\roman{enumi}}\renewcommand{\labelenumi}{\rm{(\theenumi)}}}

\numberwithin{equation}{section}

\begin{document}

\maketitle

\begin{abstract}
In this paper we introduce $\IhJ$-convergence which is a common
generalization of the $\I^*$-convergence of sequences, double
sequences and nets. We show that many results that were shown
before for these special cases are true for the
$\IhJ$-convergence, too.


\noindent Mathematical Reviews subject classification: Primary 54A20, 40A05; Secondary 40B05.
\end{abstract}

\maketitle

\sectionmark

\section{Historical background and introduction}

The main topic of this paper is convergence of a function along an
ideal. As the dual notion of the convergence along a filter was
studied as well, let us start by saying a few words about the
history of this concept.

It was defined for the first time probably by Henri Cartan
\cite{CARTANULTRA} (see also \cite[p.71, Definition
1]{BOURBAKIGTENG}).
Although the notion of a limit along a filter
was defined here in the maximal possible generality -- the
considered filter could be a filter on an arbitrary set and the
limit was defined for any map from this set to a topological space
-- the attention of mathematicians in the following years was
mostly focused to two special cases.

In general topology the notion of the limit of a filter on a
topological space $X$ became  one of the two basic tools used to
describe the convergence in general topological spaces together
with the notion of a net (see \cite[Section 1.6]{ENGNEW}).

Some authors studied also the convergence of a sequence along a
filter. This notion was rediscovered independently by several
authors, we could mention A.~Robinson \cite{ROBINSONLIMITS},
A.~R.~Bernstein \cite{BERNSTEINFILT} (these authors used
ultrafilters only) or M.~Kat\v{e}tov \cite{KATETOVFIL}.

The definition of the limit along a filter can be reformulated
using ideals -- the dual notion to the notion of filter. This type
of limit of sequences was introduced independently by P.~Kostyrko,
M.~Ma\v{c}aj and T. \v{S}al\'at \cite{KMS} and F.~Nuray and
W.~H.~Ruckle \cite{NURAYRUCKLE} and studied under the name
\emph{$\I$-convergence} of a sequence by several authors (see also
\cite{DEMIRCILIMSUP,KMSS,KSW}). The motivation for this direction
of research was an effort to generalize some known results on
statistical convergence. Since the notions that we intend to
generalize in this paper stem from one of
the results on the statistical convergence, let us describe in
more detail how they evolved.

Motivated by a result of T.~\v{S}al\'at \cite{SALATSTAT} and
J.~A.~Fridy \cite{FRIDY} about statistically convergent sequences,
the authors of \cite{KMS} also defined so called
$\I^*$-convergence (a sequence $\seq xn$ being
\emph{$\I^*$-convergent} to $x$ provided that there exist
$M\in\FI$ such that the corresponding subsequence converges to
$x$) and asked for which ideals the notions of $\I$-convergence
and $\I^*$-convergence coincide. This question was answered in
\cite{KSW} where the authors showed that these notions coincide
\iaoi the ideal $\I$ satisfies the property AP, which we call
$\APIFin$ here (see also \cite{KMSS,NURAYRUCKLE}).

Later the analogues of the notion of $\I^*$-convergence were
defined and similar characterizations were obtained for double
sequences (see \cite{DASKOSMAWI,KUMARDOUBLE}) and nets (see
\cite{LAHDASNETS}).

In this paper we define $\IhJ$-convergence as a common generalization of all these types of
$\I^*$-convergence and obtain results which strengthen the results from the above papers. In
the last section we also point at neglected relation between the $\I$-convergence of sequences and
double sequences.

Although our motivation arises mainly from the results obtained
for sequences, we will work with functions. One of the reasons is
that using functions sometimes helps to simplify notation. Another
reason is that we tried to obtain the maximal possible generality
allowed by the tools we are using.

\section{Notation and preliminaries}

In this section we recall some notions and results concerning the
$\I$-convergence.

If $S$ is a set, then a system $\I\subseteq\powerset{S}$ is called
an \emph{ideal on $S$} if it is additive, hereditary and
non-empty, that is,
\begin{compactenum}
\enu
  \item $\emps\in\I$,
  \item $A,B\in\I$ \Ra $A\cup B\in\I$,
  \item $A\in\I$ $\land$ $B\subseteq A$ \Ra $B\in\I$.
\end{compactenum}
An ideal on $S$ is called \emph{admissible} if it contains all
singletons, that is, $\{s\}\in\I$ for each $s\in S$. An ideal $\I$ on
$S$ is called \emph{proper} if $S\notin\I$, a proper ideal is
called \emph{maximal} if it is a maximal element of the set of all
proper ideals on $S$ ordered by inclusion. It can be shown that a
proper ideal $\I$ is maximal \iaoi $(\forall A\subseteq S)$
$A\in\I$ $\lor$ $S\sm A\in \I$.

We will denote by $\Fin$ the ideal of all finite subsets of a
given set $S$.

The dual notion to the notion of an ideal is the notion of a
filter. A system $\mc F\subseteq\powerset S$ of subsets of $S$ is
called a \emph{filter on $S$} if
\begin{compactenum}
\enu
  \item $S\in\mc F$,
  \item $A,B\in\mc F$ \Ra $A\cap B\in\mc F$,
  \item $A\in\mc F$ $\land$ $B\supseteq A$ \Ra $B\in\mc F$.
\end{compactenum}
A filter $\mc F$ is called \emph{proper} if $\emps\notin\mc F$.

The dual notion to the notion of a maximal ideal is the notion of
\emph{ultrafilter.}

A system $\mc B\subseteq\powerset S$ is called \emph{filterbase} if
\begin{compactenum}
\enu
  \item $\mc B\ne\emps$,
  \item $A,B\in\mc B$ \Ra $(\exists C\in\mc B)$ $C\subseteq A\cap B$.
\end{compactenum}
If $\mc B$ is a filterbase, then the system
$$\mc F=\{A\supseteq B;B\in \mc B\}$$
is a filter. It is called filter \emph{generated} by the base $\mc B$.

For any ideal $\I$ on a set $S$ the system
$$\FI=\{X\sm A; A\in \I\}$$
is a filter on $S$. It is called the \emph{filter associated with
the ideal $\I$.} In a similar way we can obtain ideal from any
filter. This yields a one-to-one correspondence between ideals and
filters on a given set.

\begin{DEF}\label{DEFICONV}
Let $\I$ be an ideal on a set $S$ and $X$ be a topological space.
A function $\Zobr fSX$ is said to be \emph{$\I$-convergent to
$x\in X$} if
$$\Invobr fU=\{s\in S; f(s)\in U\} \in \FI$$
holds for every neighborhood $U$ of the point $x$.

We use the notation
$$\Ilim f = x.$$
\end{DEF}

If $S=\N$ we obtain the usual definition of $\I$-convergence of
sequences. In this case the notation $\Ilim x_n=x$ is used.

We include a few basic facts concerning $\I$-convergence for
future reference.
\begin{LM}\label{LMCOMPAR}
Let $S$ be a set, let $\I$, $\I_1$ and $\I_2$ be ideals on $S$ and
let $X$ and $Y$ be topological spaces.
\begin{enumerate}
\enu
  \item\label{itIMPROPER} If $\I$ is not proper, that is, if $\I=\powerset
    S$, then every function $\Zobr fSX$ converges to each point of $X$.
  \item\label{itCOMPAR} If $\I_1\subseteq\I_2$, then for every function $\Zobr
    {f}SX$, we have
    $$\Jlim{\I_1} f = x \qquad \text{implies} \qquad \Jlim{\I_2} f=x.$$
  \item\label{itHAUS} If $X$ is Hausdorff and $\I$ is proper,
     then every function $\Zobr fSX$ has at most one $\I$-limit.
  \item\label{itCONTIN} If $\Zobr gXY$ is a continuous mapping and
    $\Zobr fSX$ is $\I$-convergent to $x$, then $g\circ f$ is $\I$-convergent to $g(x)$.
  \item\label{itCOMPACT} If $\I$ is a maximal ideal and $X$ is compact,
    then every function $\Zobr fSX$ has an $\I$-limit.
\end{enumerate}
\end{LM}

Let us note that the above properties are more frequently stated
for filters rather than ideals. Moreover, the property
\eqref{itHAUS} is in fact a characterization of Hausdorff spaces
and the property \eqref{itCOMPACT} is a characterization of
compact spaces.

\section{$\IhJ$-convergence}

\subsection{Definition and basic results}

As we have already mentioned, we aim to generalize the notion of
$\I^*$-convergence of sequences, introduced in \cite{KMS} for
sequences of real numbers and generalized to metric spaces in
\cite{KSW}. Since we are working with functions, we modify this
definition in the following way:
\begin{DEF}\label{DEFIHCONV}
Let $\I$ be an ideal on a set $S$ and let $\Zobr fSX$ be a
function to a topological space $X$. The function $f$ is called
\emph{$\I^*$-convergent} to the point $x$ of $X$ if there exists a
set $M\in\FI$ such that the function $\Zobr gSX$ defined by
$$g(s)=
\begin{cases}
f(s), &\text{if }s\in M\\
x, &\text{if }s\notin M
\end{cases}$$
is $\Fin$-convergent to $x$. If $f$ is $\I^*$-convergent to $x$,
then we write $\Ihlim f=x$.
\end{DEF}
The usual notion of $\I^*$-convergence of sequences is a special
case for $S=\N$. Similarly as for the $\I$-convergence of
sequences, we write $\Ihlim x_n=x$.

In fact, the $\I^*$-convergence was defined in \cite{KMS} in a
slightly different way -- the $\Fin$-convergence of the
restriction $g|_M$ was used. It is easy to see that these two
definitions are equivalent. Our approach will prove advantageous
when using more complicated ideals instead of $\Fin$.

In the definition of $\IhJ$-convergence we simply replace the
ideal $\Fin$ by an arbitrary ideal on the set $S$.
\begin{DEF}\label{DEFIHJ}
Let $\J$ and $\I$ be ideals on a set $S$, let $X$ be a topological
space and let  $x$ be an element of $X$. The function $\Zobr fSX$
is said to be \emph{$\IhJ$-convergent} to $x$ if there exists a
set $M\in\FI$ such that the function $\Zobr gSX$ given by
$$g(s)=
\begin{cases}
f(s), &\text{if }s\in M\\
x, &\text{if }s\notin M
\end{cases}$$
is \J-convergent to $x$. If $f$ is $\IhJ$-convergent to $x$, then
we write $\IhJlim f=x$.
\end{DEF}

As usual, in the case $S=\N$ we speak about $\IhJ$-convergence of
sequences and use the notation $\IhJlim x_n=x$.

\begin{REM}\label{REMDECOMP}
The definition of $\IhJ$-convergence can be reformulated in the
form of decomposition theorem. A function $f$ is $\IhJ$-convergent
\iaoi it can be written as $f=g+h$, where $g$ is $\J$-convergent
and $h$ is non-zero only on a set from $\I$. An analogous
observation was made in \cite{CONNORSTRONG} for the statistical
convergence of sequences and in \cite{MORICZDOUBLE} for the
statistical convergence of double sequences.
\end{REM}

\begin{REM}\label{REMTRACE}
A definition of $\IhJ$-convergence following more closely the
approach from \cite{KMS} would be: there exists $M\in\FI$ such
that the function $f|_M$ is $\J|M$-convergent to $x$
where $\J|M=\{A\cap M; A\in\J\}$ is
the trace of $\J$ on $M$. These two
definitions are equivalent but the one given in Definition
\ref{DEFIHJ} is somewhat simpler.
\end{REM}

One can show easily directly from the definitions that
$\J$-convergence implies $\IhJ$-convergence.
\begin{LM}\label{LMKRAIK}
If $\I$ and $\J$ are ideals on a set $S$ and $\Zobr fSX$ is a
function such that $\Jlim{\J} f=x$, then $\IhJlim f=x$.
\end{LM}

Using Lemma \ref{LMCOMPAR} \eqref{itCOMPAR} and the definition of
$\IhJ$-convergence we get immediately
\begin{PROP}\label{PROPCOMPAR}
Let $\I$, $\I_1$, $\I_2$, $\J$, $\J_1$ and $\J_2$ be ideals on a set $S$ such that
$\I_1\subseteq\I_2$ and $\J_1\subseteq\J_2$ and let $X$ be a topological space. Then for any
function $\Zobr fSX$ we have
\begin{gather*}
\IhJhlim{\I_1}{\J} f=x \qquad \Ra \qquad \IhJhlim{\I_2}{\J} f=x,\\
\IhJhlim{\I}{\J_1} f=x \qquad \Ra \qquad \IhJhlim{\I}{\J_2} f=x.
\end{gather*}
\end{PROP}

In what follows we are going to study the relationship between the
$\I$-convergence and $\IhJ$-convergence. In particular, we will
specify the conditions under which the implications
\begin{eqnarray}
\IhJlim f=x \qquad &\Ra& \qquad \Ilim f =x, \label{IMP1}\\
\Ilim f =x \qquad &\Ra& \qquad \IhJlim f=x, \label{IMP2}
\end{eqnarray}
hold.

We start with the easier implication \eqref{IMP1}. In the case
$\J=\Fin$ this implication is known to be true for the admissible
ideals, that is, for ideals fulfilling $\J\subseteq\I$. We next
show that the same is true in general.
\begin{PROP}\label{PROPIMP1}
Let $\I,\J$ be ideals on a set $S$, let $X$ be a topological space
and let $f$ be a function from $S$ to $X$.

\begin{enumerate}
\enu
  \item\label{itIMP1:1} If the implication \eqref{IMP1} holds for some point $x\in X$
    which has at least one neighborhood different from $X$,
    then $\J\subseteq\I$. Consequently, if the implication
    \eqref{IMP1} holds in a topological space that is not
    indiscrete, then $\J\subseteq\I$.
  \item\label{itIMP1:2} If $\J\subseteq\I$, then the implication \eqref{IMP1}
    holds.
\end{enumerate}
\end{PROP}

\begin{proof}
\eqref{itIMP1:1} Suppose that $\J\nsubseteq\I$, that is, there
exists a set $A\in \J\sm\I$. Let $x$ be a point with a
neighborhood $U\subsetneqq X$ and $y\in X\sm U$. Let us define a
function $\Zobr fSX$ by
$$
f(t)=
  \begin{cases}
    x & \text{if }t\notin A, \\
    y & \text{otherwise}.
  \end{cases}
$$
Clearly, $\Jlim{\J}f=x$ and thus by Lemma \ref{LMKRAIK} we get
$\IhJlim f=x$. As $\Invobr f{X\sm U}=A\notin\I$, the function $f$
is not $\I$-convergent to $x$

\eqref{itIMP1:2} Let $X$ be any topological space, $x\in X$ and
$\Zobr fSX$. Let $\J\subseteq\I$ and $\IhJlim f=x$. By the
definition of $\IhJ$-convergence there exists $M\in\FI$ such that
$$C:=\Invobr f{X\sm U}\cap M\in\J\subseteq\I$$ for each
neighborhood $U$ of the point $x$. Consequently,
$$\Invobr f{X\sm U} \subseteq (X\sm M) \cup C \in \I$$
and thus $\Ilim f=x$.
\end{proof}

\subsection{Additive property and $\IhJ$-convergence}

Inspired by \cite{KSW} and \cite{LAHDASTOP} where the case $\J=\Fin$ and $S=\N$ is
investigated, we now concentrate on an algebraic characterization of the ideals $\I$ and $\J$
such that the implication \eqref{IMP2}
holds for each function $\Zobr fSX$. Before doing this we need to prove some auxiliary
results.

\begin{DEF}
Let $\J$ be an ideal on a set $S$. We write $A\subJ B$ whenever
$A\sm B \in \J$. If $A\subJ B$ and $B\subJ A$, then we write
$A\simJ B$. Clearly,
$$A\simJ B \qquad \Lra \qquad A\triangle B \in \J.$$

We say that a set $A$ is \emph{\J-pseudointersection} of a system
$\{A_n; n\in\N\}$ if $A\subJ A_n$ holds for each $n\in\N$.
\end{DEF}

In the case $\J=\Fin$ we obtain the notion of pseudointersection
and the relations $\subseteq^*$ and $=^*$ which are often used in
set theory (see \cite[p.102]{JUSTWEESE2}).

It is easy to see that using the symbols $\subJ$ and $\simJ$ can
be understood as another way of speaking about the equivalence
classes of the subsets of $S$ in the quotient Boolean algebra
$\powerset{S}/\J$.

In the following lemma we describe several equivalent formulations
of a condition for ideals $\I$ and $\J$ which will play crucial
role in further study.

\begin{LM}\label{LMAP}
Let $\I$ and $\J$ be ideals on the same set $S$. The following
conditions are equivalent:
\begin{enumerate}
\enu
 \item\label{PI1} For every sequence $(A_n)_{n\in\N}$ of sets from \I there is $A\in\I$
    such that $A_n\subJ A$ for all $n$'s.
 \item\label{PI2} Any sequence $(F_n)_{n\in\N}$ of sets from $\FI$ has a \J-pseudointersection in
 $\FI$.
 \item\label{PI3} For every sequence
        $(A_n)_{n\in\N}$ of sets belonging to $\I$ there exists a sequence
        $(B_n)_{n\in\N}$ of sets from $\I$  such that $A_j \simJ B_j$
        for $j\in\N$ and $B=\bigcup_{j\in\N} B_j\in\I$.
 \item\label{PI4} For every sequence of mutually disjoint sets
        $(A_n)_{n\in\N}$ belonging to $\I$ there exists a sequence
        $(B_n)_{n\in\N}$ of sets belonging to $\I$ such that $A_j \simJ B_j$
        for $j\in\N$ and $B=\bigcup_{j\in\N} B_j\in\I$.
 \item\label{PI6} For every non-decreasing sequence $A_1\subseteq A_2 \subseteq \dots \subseteq A_n \subseteq \dots$
        of sets from $\I$ there exists a sequence
        $(B_n)_{n\in\N}$  of sets belonging to $\I$ such that $A_j \simJ B_j$
        for $j\in\N$ and $B=\bigcup_{j\in\N} B_j\in\I$.
 \item\label{PI5} In the Boolean algebra $\powerset{S}/\J$ the ideal $\I$
   corresponds to a $\sigma$-directed subset, that is, every countable subset has an
   upper bound.
\end{enumerate}
\end{LM}

Note that (\ref{PI2}) is just a dual formulation of (\ref{PI1}).
Similarly, (\ref{PI5}) is the formulation of (\ref{PI1}) in the
language of Boolean algebras. The equivalence of (\ref{PI3}),
(\ref{PI4}), (\ref{PI6}) can be easily shown by the standard
methods from the measure theory. Proof of the equivalence of the
remaining conditions is similar to the proof of Proposition 1 of
\cite{BADEKO}, where the case $\J=\Fin$ is considered. We include
the proof for the sake of completeness and also to stress that the
validity of this lemma does not depend on the countability of $S$
or the assumption that $\J\subseteq\I$.

\begin{proof}
(\ref{PI1})\Ra(\ref{PI6}) Let  $A_1\subseteq A_2 \subseteq \dots
\subseteq A_n \subseteq \dots$ be a non-decreasing sequence of
sets from $\I$. Since each $A_n\in\I$, the condition (\ref{PI1})
yields the existence of a set $A\in\I$ satisfying $A_n\subJ A$ for
$n\in\N$. Let $B_n:=A\cap A_n$. Since $B_n\subseteq A$, we have
$B_n\in\I$. Moreover, $B_n\triangle A_n = A_n\sm A \in \J$, thus
$B_n\simJ A_n$. Finally, $B=\bigcup_{j\in\N} B_j\subseteq A\in\I$,
as required.

(\ref{PI3})\Ra(\ref{PI1}) Let  $(A_n)_{n\in\N}$ be a sequence of
sets belonging to $\I$. By (\ref{PI3}) there exists a sequence
$(B_n)_{n\in\N}$ of sets from $\I$ such that for all $n$ we have
$B_n \simJ A_n$ and $A:=\bigcup_{n\in\N} B_n \in \I$. From
$A_n\triangle B_n \in \J$ and $B_n\subseteq A$ we get $A_n\subJ
A$, which proves (\ref{PI1}).
\end{proof}

It is also easy to see that in  condition (\ref{PI2}) it suffices to consider only
sequences of sets from a filterbase. This reformulation of (\ref{PI2}) can be sometimes
easier to prove.

\begin{DEF}
Let $\I$, $\J$ be ideals on a set $S$. We say that $\I$ has the
\emph{additive property} with respect to $\J$, or more briefly
that $\APIJ$ holds, if any of the equivalent conditions of Lemma
\ref{LMAP} holds.
\end{DEF}

The condition AP from \cite{KSW}, which characterizes
ideals such that $\I^*$-convergence implies $\I$-convergence, is equivalent to the condition
$\APIFin$. Let us note that ideals fulfilling this condition are
often called \emph{P-ideals} (see for example \cite{BADEKO} or
\cite{FILIPOWBOLZ}).

In the following two theorems we show that the
condition $\APIJ$ is the correct generalization of conditions AP
from \cite{KSW}, \cite{LAHDASTOP} and \cite{DASKOSMAWI}. In
particular, as special cases of our results we obtain Theorem 3.1 of \cite{KSW},
Theorem 8 of \cite{LAHDASNETS} and Theorem 2 of \cite{DASKOSMAWI}.

Although we do not consider arbitrary topological spaces,
we feel that the restriction to the first countable spaces is
sufficient for most applications. For example, in \cite{KSW} the
authors work only with metric spaces and in \cite{LAHDASTOP} the
case that $X$ is a first countable $T_1$-space is considered.

\begin{THM}\label{THMIMP2}
Let $\I$ and $\J$ be ideals on a set $S$ and let $X$ be a first
countable topological space. If the ideal $\I$ has the additive
property with respect to $\J$, then for any function $\Zobr fSX$
the implication \eqref{IMP2} holds. In other words, if the
condition $\APIJ$ holds, then the $\I$-convergence implies the
\IhJ-convergence.
\end{THM}

\begin{proof}
Let $\Zobr fSX$ be an \I-convergent function and let $x=\Ilim f$.
Let $\mc B=\{U_n; n\in\N\}$ be a countable base for $X$ at the
point $x$. By  the \I-convergence of $f$ we have
$$\Invobr f{U_n} \in \FI$$
for each $n$, thus by Lemma \ref{LMAP} there exists $A\in\FI$ with
$A\subJ \Invobr f{U_n}$, that is, $A\sm\Invobr f{U_n}\in\J$ for
all $n$'s.

Now it suffices to show that the function $\Zobr gSX$ given by
$g|_A=f|_A$ and $\Obr g{S\sm A}=\{x\}$ is $\J$-convergent to $x$.
As for $U_n\in\mc B$ we have
$$\Invobr g{U_n}= (S\sm A)\cup \Invobr f{U_n}= S\sm(A\sm\Invobr f{U_n}),$$
and the set $A\sm\Invobr f{U_n}$ belongs to $\J$, its complement
$\Invobr g{U_n}$ lies in $\FJ$, as required.
\end{proof}

Let us recall that a topological space $X$ is called
\emph{finitely generated space} or \emph{Alexandroff space} if any
intersection of open subsets of $X$ is again an open set (see
\cite{ARENAS}). Equivalently, $X$ is finitely generated \iaoi each
point of $x$ has a smallest neighborhood. Finitely generated
$T_1$-spaces are precisely the discrete spaces.

\begin{THM}\label{THMIMP2b}
Let $\I$, $\J$ be ideals on a set $S$ and let $X$ be a first
countable topological space which is not finitely generated. If
the implication \eqref{IMP2} holds for any function $\Zobr fSX$,
then the ideal $\I$ has the additive property with respect to
$\J$.
\end{THM}

\begin{proof}
Let $x\in X$ be an accumulation point of $X$ which does not have
a smallest neighborhood. Let $\mc B=\{U_i; i\in\N\cup\{0\}\}$ be
a countable base at $x$ such that $U_{n}\supsetneqq U_{n+1}$ and
$U_0=X$. Suppose we are given some countable family $A_n$ of
mutually disjoint sets from $\I$.

For each $n\in\N$ choose an $x_n\in U_{n-1}\sm U_{n}$. Let us
define $\Zobr fSX$ as
$$f(s)=
  \begin{cases}
    x_n & \text{if } s\in A_n, \\
    x & \text{if } s\notin\bigcup_{n\in\N} A_n.
  \end{cases}
$$

We have $\Invobr f{X\sm U_n}=\bigcup_{i=1}^n A_i\in\I$, hence
$\Ilim f=x$. By the assumption, $\IhJlim f=x$, which means that
there is $A\in\FI$ such that the function $\Zobr gSX$ given by
$g|_A=f|_A$ and $\Obr g{S\sm A}=\{x\}$ is $\J$-convergent to $x$.
This yields
$$\Invobr g{X\sm U_n}=\left(\bigcup_{i=1}^n A_i\right)\cap A = \bigcup_{i=1}^n (A_i\cap A)\in \J.$$
From this we have $A_i\cap A\in \J$, thus $B_i:=A_i\sm A \simJ A_i$.

Note that, at the same time
$$\bigcup_{i\in\N} B_i = \left(\bigcup_{i\in\N} A_i\right)\sm A \subseteq S\sm A\in \I.$$
We have shown \eqref{PI4} from Lemma \ref{LMAP}.
\end{proof}

\begin{REM}\label{REMIMP2bLOCAL}
Let us note that we have in fact proved a slightly stronger result: Whenever $x$ is an
accumulation point of $X$ such that there exists a countable basis at $x$, the point $x$ does
not have a smallest neighborhood and the implication \eqref{IMP2} holds for each function
$\Zobr fSX$ which is $\IhJ$-convergent to $x$, then the ideal $\I$ has the additive property with respect to $\J$.
\end{REM}

We next provide an example showing that Theorem \ref{THMIMP2} does
not hold in general for spaces which are not first countable.

\begin{EXA}\label{EXAJASREC}
Pointwise $\I$-convergence of sequences of continuous real functions was studied in \cite{JASINSKIRECLAW} and \cite{JASINSKIRECLAW2008}.
It can be understood as convergence of sequences of elements of the space $C_p(X)$ of all real continuous function endowed with the
topology of pointwise convergence. The authors of \cite{JASINSKIRECLAW,JASINSKIRECLAW2008} defined and studied the \I-convergence property which,
using our terminology, can be formulated as follows: A topological space $X$ has the \emph{$\I$-convergence property} if \eqref{IMP2} holds in the space
$C_p(X)$ for $S=\N$ and $\J=\Fin$.

It is known that $C_p(X)$ is first countable \iaoi $X$ is
countable, see \cite[Theorem 4.4.2]{MCCOYNTANTU}. Hence our
Theorem \ref{THMIMP2} yields that all countable spaces have the
$\I$-convergence property for every P-ideal $\I$. The same result was
obtained in \cite[Corollary 1]{JASINSKIRECLAW}.

It was shown in \cite{JASINSKIRECLAW2008} that $\R$ does not have
$\I$-convergence property for any nontrivial analytic P-ideal on $\N$.
(By trivial ideals we mean the ideals of the form $\I_C=\{A\subseteq\N; A\subseteq^* C\}$
for some $C\subseteq\N$.)
Hence, $C_p(\R)$ provides the desired counterexample, which works
for a large class of ideals on $\N$.
The definition of analytic ideals, more related results
and many examples of analytic P-ideals can be found, for example,
in \cite{FARAHMEMOIRS,FILIPOWBOLZ}.
\end{EXA}

To find a counterexample showing that Theorem \ref{THMIMP2b} is in general not true without
the assumption that the space $X$ is first countable we can use any space in which all
$\I$-convergent sequences are, in some sense, trivial.

\begin{EXA}\label{EXAIMP2b}
Let us recall that $\omone$ denotes the first uncountable ordinal
with the usual ordering. Let $X$ be the topological space on the
set $\omone \cup \{\omone\}$ with the topology such that all
points different from $\omone$ are isolated and the base at the
point $\omega_1$ consists of all sets $U_\alpha=\{\beta\in X;
\beta>\alpha\}$ for $\alpha<\omone$. Notice that if
$C\subseteq\omone$ is a set such that $\omone\in\ol C$, then
$\abs{C}=\alone$.

Now let $\I$ be an admissible ideal on $\N$ and let a function
$\Zobr f{\N}X$ be $\I$-convergent to $\omone$. We will show that
then there exists $M\in\FI$ such that $f(x)=\omega_1$ for each
$x\in M$, that is, $f|_M$ is constant. Clearly, this implies that
$f$ is $\I^*$-convergent.

For the sake of contradiction, suppose that each set $M\in\FI$
contains some point $m$ such that $f(m)\ne\omone$. Since $\Invobr
fU\in\FI$, for any neighborhood $U$ of $\omone$ in $X$ there
exists $m\in\N$ with $f(m)\in U\sm\{\omone\}$. Therefore for the
set $C=\{m\in\N; f(m)\ne\omone\}$ we have $\omone\in \ol{\Obr
fC}$. Since $\Obr fC \subseteq\omone$ and it is a countable set
contained in $\omone$, this is a contradiction.

Now, by choosing an ideal $\I$ which does not have the additive
property $\APIFin$ we obtain the desired counterexample.
\end{EXA}

\section{Examples and applications}

We have already mentioned that our motivation for definition and study of $\IhJ$-convergence
was an effort to provide a common generalization to the notion of $\I^*$-convergence which
was defined first for the usual sequences in \cite{KMS} and later generalized for sequences
of functions, double sequences and nets in \cite{GEZERKARAKUS}, \cite{KUMARDOUBLE} and
\cite{LAHDASNETS}, respectively.

In this section we show that the notion of the $\IhJ$-convergence is a correct generalization
of these notions, that is, all these notions are special cases of the $\IhJ$-convergence. We
begin with the notion of $\I^*$-convergence of double sequences.

\subsection{Double sequences}

In the study of double sequences several types of convergence are used. For our purposes, the
following one is the most important.

\begin{DEF}[\cite{BALDEMS,PRINGSHEIM}]\label{DEFPRG}
A double sequence $\dseq xmn$ of points of a topological space $X$
is said to converge to $x$ \emph{in Pringsheim's sense} if for
each neighborhood $U$ of the point $x$
$$(\exists k\in\N) (\forall m\geq k) (\forall n\geq k) x_{m,n}\in U.$$
\end{DEF}

It is easy to see that the convergence in Pringsheim's sense is equal to
the $\I$-convergence along the \emph{Pringsheim's ideal}
$\Iprg$ on $\N\times\N$
whose dual filter $\FIh{\Iprg}$ is given by the filterbase
$$\Bprg=\{\intrvl{m}\infty\times\intrvl{m}\infty; m\in\N\}.$$
We will give a different description of this ideal in Example \ref{EXAMACAJID}.

Altogether four types of convergence of double sequences were
studied in \cite{BALDEMS}. All of them can be described as
$\I$-convergences using appropriate ideals on $\N\times\N$ (see
Figure \ref{FIGIDEALS}). In fact, we denote the Pringsheim's ideal by
$\Iprg$ in order to be consistent
with the notation of \cite{BALDEMS}.

\begin{figure}
\begin{center}
\begin{tabular}{cccc}
\includegraphics{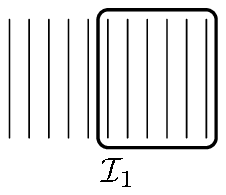} & \includegraphics{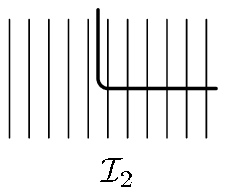} & \includegraphics{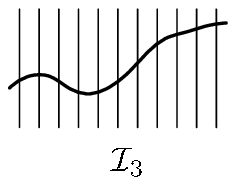}
& \includegraphics{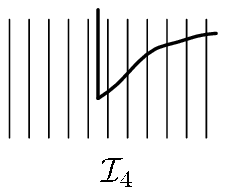}
\end{tabular}
\end{center}
\caption{Ideals from \cite{BALDEMS} illustrated by depicting
typical sets from the filterbase. Vertical lines represent the
partition of $\N\times\N$ into countably many infinite sets
$\{i\}\times\N$.}\label{FIGIDEALS}
\end{figure}

The $\I^*$-convergence of double sequences studied in \cite{KUMARDOUBLE} and
\cite{DASKOSMAWI} is the same as $\IhJh{\I}{\Iprg}$-convergence in $\N\times\N$.
Therefore, as a special case of our Theorems
\ref{THMIMP2} and \ref{THMIMP2b} for $S=\N\times\N$ and $\J=\Iprg$
we obtain Proposition 4.2 of \cite{KUMARDOUBLE}, and Theorems 3 and 4
of \cite{DASKOSMAWI}. Note that in \cite{KUMARDOUBLE} and
\cite{DASKOSMAWI} only the ideals containing $\Iprg$ are
considered, see Proposition \ref{PROPIMP1}.

\subsection{Further examples}

In order to avoid technical details we will define neither the notions of pointwise and
uniform $\I^*$-convergence of a sequence of functions defined in \cite{GEZERKARAKUS}, nor
the notions of the $\I$- and $\I^*$-convergence of nets defined in \cite{LAHDASNETS}.

We just mention that, given an ideal $\mc L$ on $\N$, the uniform $\mc L^*$-convergence of
a sequence of functions defined on $X$ is precisely the $\IhJ$-convergence for the ideal $\I$
on $X\times\N$ given by the filterbase $\{X\times (\N\sm A); A\in\mc L\}$ and the ideal $\J$
given by the filterbase $\{X\times (\N\sm A); A\in\Fin\}$. The pointwise $\mc
L^*$-convergence can be obtained if $\I$ is the ideal of all sets $A\subseteq X\times\N$ such
that for each $x\in X$ the $x$-cut $A_x:=\{n\in\N; (x,n)\in A\}$ belongs to $\mc L$, and $\J$
consists of all sets such that each $A_x$ is finite.

In both cases it can be shown that the condition $\APIJh{\I}{\J}$ is equivalent to the condition $\APIJh{\mc{L}}{\Fin}$.
Hence our Theorems \ref{THMIMP2} and \ref{THMIMP2b} imply that these two types of $\I$-convergence
are equivalent to corresponding $\I^*$-convergence \iaoi $\APIJh{\mc{L}}{\Fin}$ holds. This observation has been made already in \cite{GEZERKARAKUS}.

Similarly, the concept of $\I^*$-convergence of nets is a special case of $\IhJ$-convergence
and Theorem 12 of \cite{LAHDASNETS} can be obtained from our Theorems \ref{THMIMP2} and
\ref{THMIMP2b} by choosing the section filter of the considered directed set for $\J$ (the
definition of the section filter
can be found, for example, in \cite[p.60]{BOURBAKIGTENG}).

\subsection{$\I$-convergence of double sequences}

We close this paper with an observation concerning the
$\I$-convergence of double sequences.

Notice that any bijection between sets $S$ and $T$
naturally gives rise to a bijection between $X^S$ and $X^T$, an
isomorphism between Boolean algebras $\powerset{S}$ and
$\powerset{T}$ and also to an isometric isomorphism between linear
normed spaces $\ell_\infty(S)$ and $\ell_\infty(T)$.
It is easy to see that this correspondence
preserves also the properties related to the notion of
$\I$-convergence. Hence results about $\I$-convergence for a given
set $S$ do not depend on the natural (partial) ordering on the set $S$ in any way.
Thus these results can be transferred to any set of the same cardinality.

We can use any bijection between $\N$ and
$\N\times\N$ to relate results about sequences and double
sequences. It is interesting to note that several authors working
in this area did not realize this possibility.

The basic results on $\I$-convergence (such as additivity,
multiplicativity, uniqueness of limit in Hausdorff spaces) need
not be shown again for double sequences, since they follow from
the analogous result for sequences; although the proofs are rather
trivial in both cases. But there are also some more interesting
concepts that were defined for double sequences in a such way that
they are preserved by this correspondence. Namely, this is true
for the notions of $\I$-Cauchy double sequences, extremal
$\I$-limit points ($\I$-limit superior and $\I$-limit inferior)
and $\I$-cluster points.

In this way, some results from the papers
\cite{DASMALIKEXTREMALDOUBLE,GURDALSAHINEREXPTS,KUMARICORE,TRIPDS}
on the above mentioned concepts can be obtained from the results
of \cite{DEMIRCILIMSUP,DEMSICAUCH,KMSS,LAHDASLIMSUP}. Actually,
the fact that a double sequence is $\I$-convergent \iaoi it is
$\I$-Cauchy is shown in Proposition 5 of \cite{DEMSICAUCH} using a
bijection between $\N$ and $\N\times\N$.

The above observation can also be used to get an alternative description
the ideal $\Iprg$.

\begin{EXA}\label{EXAMACAJID}
A basic example of an ideal which does not have the property $\APIFin$ is the ideal $\Imac$
given in Example 1.1.(g) of \cite{KSW} and Example (XI) of \cite{KMS}. It is defined as
follows: Suppose we are given any partition $\N=\bigcup_{i=1}^n D_i$ of $\N$ into countably
many infinite sets. A set $A\subseteq\N$ belongs to $\Imac$ \iaoi it intersects only finitely
many $D_i$'s.
Of course, choosing different partitions of $\N$ can lead to ideals which are different, but
equivalent from the point of view of $\I$-convergence.

We can also use any countable set instead of $\N$.
In particular, as observed in
the proof of Corollary 4 in \cite{BALDEMS}, by choosing the
partition of $\N\times\N$ into sets $D_i=\{(n,i);n\ge
i\}\cup\{(i,k); k\ge i\}$ we obtain the ideal $\Iprg$ in this way.
Similarly, by using $D_i=\{i\}\times\N$ we get the ideal $\Iuni$ of
\cite{BALDEMS} (see Figure \ref{FIGDECOMP}).
Thus the ideals $\Iuni$, $\Iprg$ and $\Imac$ are essentially the same. In particular, this
gives an alternative proof that $\APIJh{\Iprg}{\Fin}$ and $\APIJh{\Iuni}{\Fin}$ fail, see
\cite{DASKOSMAWI}.
\end{EXA}

\begin{figure}
\begin{center}
\begin{tabular}{cccc}
\includegraphics{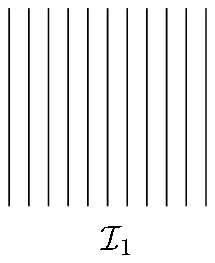} & \includegraphics{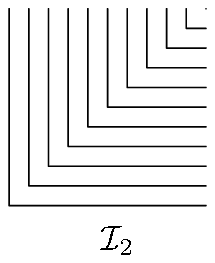}
\end{tabular}
\end{center}
\caption{Partitions of $\N\times\N$ defining the ideals $\Iuni$ and
$\Iprg$}\label{FIGDECOMP}
\end{figure}

\medskip
\noindent {\bf Acknowledgment}. We would like to thank the referees for
suggesting several improvements and corrections. In particular,
one of the referees pointed our attention to results of
\cite{JASINSKIRECLAW2008}, which were used to make Example
\ref{EXAJASREC} more general.


\end{document}